\documentclass{amsart}
\usepackage{amsfonts,amssymb,amsmath,amsthm}
\usepackage{url}

\newtheorem{theorem}{Theorem}[section]
\newtheorem{lemma}[theorem]{Lemma}

\newtheorem{corollary}[theorem]{Corollary}
\theoremstyle{definition}
\newtheorem{definition}[theorem]{Definition}

\author{Radu-B. Munteanu}
\address{
University of Bucharest\\
Bd. Regina Elisabeta 4-12 Sector 1\\ Bucharest, Romania}
\email{radu-bogdan.munteanu@g.unibuc.ro}
\thanks{This work was done while the author was Ph.D. student at the University of Ottawa}
\begin{document}
\title[A non-product type transformation satisfying property A]{A non-product type non-singular transformation which satisfies Krieger's property A}
\begin{abstract}
We show that there exists an ergodic non-singular transformation
which satisfies Krieger's property A, but which is not of product
type.
\end{abstract}
\maketitle

\section{Introduction}
Property A was introduced by Krieger \cite{K2} in order to show that
there exist ergodic non-singular transformations which are not orbit
equivalent to any product odometer. A non-singular transformation
that is equivalent to a product odometer is said to be of product
type. Krieger \cite{K2}, \cite{K3} showed that property A is an
invariant for orbit equivalence and that any product odometer of
type $III$ satisfies this property. He also constructed an ergodic
non-singular transformation of type $III$ that does not satisfies
this property and therefore it is not of product type. We recall
that non-singular transformations of type $III_{\lambda}$,
$\lambda\neq 0$, are unique up to orbit equivalence and they are of
product type.

In order to characterize the ITPFI factors among AFD factors of type
$III_{0}$, Connes and Woods [CW] introduced a property of ergodic
flows called approximate transitivity, or shortly, AT. They proved
that, up to isomorphism, an AFD factor of type $III_{0}$ is ITPFI if
and only its flow of weights is aperiodic and approximately
transitive. As any ITPFI factor is the Krieger factor of a product
odometer, and there is a bijective correspondence between the
isomorphism classes of AFD factors of type $III_{0}$ and the orbit
equivalence classes of ergodic non-singular transformations of type
$III_{0}$, their result says that an ergodic non-singular
transformation of type $III_{0}$ is of product type if and only if
its associated flow is aperiodic and AT.

In this paper we show that there exist an ergodic non-singular
transformation of type $III_{0}$ and which is not of product type
but satisfies property A. In this way, we answer a question of
Dooley and Hamachi \cite{DH}. In Section 2, we consider the AFD
factor that is not an ITPFI factor, studied by Giordano and
Handelman \cite{GH}. This factor can be realized as the factor
associated to a countable measured ergodic equivalence relation.  We
find an explicit non-singular transformation $T$ of type $III_{0}$
which determines this equivalence relation up to orbit equivalence,
and which, by \cite{GH}, is not of product type. In Section 3, we
show that $T$ does not satisfy Krieger's property A.

\section{Prelimnaries}
Let $(X,\mathfrak{B},\mu)$ be a Lebesgue space. A one to one
bi-measurable mapping $T:X\rightarrow X$ is non-singular if it
preserves null sets. Furthermore, $T$ is ergodic if $T^{-1}(A)=A$
implies either $\mu(A)=0$ or $\mu(X\setminus A)=0$. For a
non-singular transformation $T$, we denote by $[T]$ the full group
of $T$, that is the group of non-singular transformations $S$
defined by measurable integer valued functions $n(x)$ as
$S(x)=T^{n(x)}$. We recall that $T$ is said to be of type $III$ if
there is no $\sigma$-finite $T$-invariant measure $\nu$ equivalent
to $\mu$. According to Krieger's ratio set, the non-singular
transformations of type $III$ are classified in subtypes
$III_{\lambda}$, $0\leqslant\lambda\leqslant 1$.

If $T'$ is another non-singular transformation on a Lebesgue space
$(X',\mathfrak{B}',\mu')$, we say that $T$ and $T'$ are orbit
equivalent if there exists a bi-measurable one to one mapping $S$
from $X$ onto $X'$ that carries $\mu$-null sets onto $\mu'$-null
sets and $\{S(T^{n}(x)); n\in\mathbb{Z}\}=\{T'^{n}(S(x)),
n\in\mathbb{Z}\}$ for $\mu$-almost all $x\in X$.

For a sequence of positive integers $(k_{n})_{n\geqslant 1}$, we
consider $X=\prod_{n\geqslant1}\{0,1,\ldots k_{n}-1\}$ endowed with
the product topology and the corresponding Borel structure. We
define $T$ on $X$ by
\begin{equation*}T(x)_{n}=\left\{\begin{array}{c}
                              0 \ \ \ \ \ \ \text { if }n<N(x) \\
                              x_{n}+1\text{ if }n=N(x) \\
                              x_{n} \ \ \ \ \ \text{ if }n> N(x).
\end{array}
\right.
\end{equation*}
where $N(x)=\min\{n\geqslant 1: x_{n}< k_{n}-1 \}$. Note that $T$ is
a non-singular and ergodic transformation. Such a transformation is
called a product odometer on $X$.

A non-singular transformation is called of product type if it is
orbit equivalent to a product odometer. All product type
non-singular transformation satisfies Krieger's Property A, which is
defined as follows.

\begin{definition} \cite{K2} A non-singular transformation $T$ on
$(X,\mathfrak{B},\mu)$ is said to satisfy Property A if there exist
constants $\eta, \delta > 0$ and a $\sigma$-finite measure
$\nu\sim\mu$ such that every set $A$ of positive measure contains a
measurable subset $B$ of positive measure and
\begin{align*}
\limsup_{s\rightarrow\infty}\nu(K_{\nu,T}(B, s,\zeta))>\eta\mu(B),
\end{align*}where
\begin{align*}
K_{\nu,T}(B, s,\zeta)&=\{x\in B, \exists \ \gamma\in [T], \
\gamma x\in B\text{ and }\\
&\log\frac{d\mu\circ\gamma}{d\mu}(x)\in
(e^{s-\delta},e^{s+\delta})\cup(-e^{s+\delta}, -e^{s-\delta})\}
\end{align*}
\end{definition}

In a natural way, a non-singular transformation $T$ on
$(X,\mathfrak{B},\mu)$, can induce an equivalence relation on $X$.
This equivalence relation will be denoted $\mathcal{R}_{T}$ and will
be given by
\[(x,y)\in\mathcal{R}_{T}\text{ if and only if }y=T^{n}x\text{ for some }n\in\mathbb{Z}\]
We immediately observe that $\mathcal{R}_{T}$ is a measurable subset
of $X\times X$, for any $x\in X$ the orbits
$\mathcal{R}_{T}(x)=\{y\in X; (x,y)\in \mathcal{R}_{T}\}$ are
countable and $\mathcal{R}_{T}$ is non-singular in the sense that
the saturation of any set of measure zero has measure zero.

An equivalence relation $\mathcal{R}$ on the Lebesgue space
$(X,\mathfrak{B},\mu)$ with these properties is called countable
measured equivalence relation, \cite{FM1}. $\mathcal{R}$ is ergodic
if for every $\mathcal{R}$-invariant set, $A\in\mathfrak{B}$
$\mu(A)=0$ or $\mu(X\setminus A)=0$. Clearly, $T$ is ergodic if and
only if $\mathcal{R}_{T}$ is ergodic. Two equivalence relations
$\mathcal{R}$ and $\mathcal{R}'$ on $(X,\mathfrak{B},\mu)$ resp. on
$(X',\mathfrak{B}',\mu')$ are called orbit equivalent if there
exists a bi-measurable one to one mapping $S$ from $X$ onto $X'$
that carries $\mu$-null sets onto $\mu'$-null such
that\[T(\mathcal{S}(x))=\mathcal{R}'(Sx)\text{ for $\mu$-almost all
}x\in X.\]Note that two non-singular transformations $T$ and $T'$ on
$(X,\mathfrak{B},\mu)$ and $(X',\mathfrak{B}',\mu')$ respectively,
are orbit equivalent if and only if $\mathcal{R}_{T}$ and
$\mathcal{R}_{T'}$ are orbit equivalent.

We denote by $M(X,\mathcal{R},\mu)$ the von Neumann algebra
associated to a countable measured equivalence relation
$\mathcal{R}$ as introduced by Feldman and Moore \cite{FM2}. This
von Neumann algebra is factor if and only if $\mathcal{R}$ is
ergodic. For $T$ an ergodic non-singular transformation on
$(X,\mathfrak{B},\mu)$, $M(X,\mathcal{R}_{T},\mu)$ is called a
Krieger factor.

\section{Construction of a non-singular transformation T}
\begin{definition} (1) A Bratteli diagram $D=(V,E)$ is a graph with
a set of vertices $V$ and a set of edges $E$, with the following
properties:
\begin{enumerate}
\item[(i)] $V$ is the disjoint union of finite subsets $V_{n}$, $n \geqslant
0$;
\item[(ii)] $E$ is the disjoint union of subsets $E_{n}$, $n\geqslant 1$, with each edge $e\in E_{n}$
connecting a vertex $s(e)\in V_{n}$ with a vertex $r(e)\in V_{n+1}$;
\item[(iii)] For every vertex $v\in V$, there exist $e\in E$ with $s(e)=v$;
\item[(iv)] For every vertex $v\in V$, except for $v\in  V_{0}$,  there exist $e \in E$ with $r(e)=v$.

\end{enumerate}
For simplicity, we assume that $V_{0}$ consists of a single vertex
$v_{0}$.

A path in $D$ is defined as a sequence $(e_{k})$ of edges with
$s(e_{1})\in V_{0}$, and, for $k\geqslant 2$,
 $s(e_{k}) = r(e_{k-1})$.
We denote by $\Omega_{n}$ the space of paths of length $n$, and by
$\Omega$
 the space of paths of infinite length. To each path of length $n$, $f=(f_{1},f_{2},\ldots
 f_{n})$, we associate the set\[Z_{f}=\{e\in\Omega, e_{k}=f_{k}, 1\leqslant k\leqslant
 n\}.\]Such a set is called cylinder of length $n$. On $\Omega$ we consider the $\sigma$-algebra generated by all cylinder sets.

An AF-measure $\mu_{p}$ on $\Omega$ is a measure determined by a
system of transition probabilities p (i.e. maps $p : E \rightarrow
[0, 1]$ with $p(e)>0$ and $\sum_{\{e\in E, s(e)=v\}}p(e)=1$ for
every $v\in V$) given by
\[\mu_{p}(f)=\prod_{k=1}^{n}p(f_{k}),\]
where $f=(f_{1},f_{2},\ldots f_{n})$ is a cylinder of length $n$.

The tail equivalence relation on $\Omega$, denoted by
$\mathcal{R}_{\Omega}$, is defined defined by
\[e\mathcal{R}_{\Omega}f\text{ if and only if for some }n,\ e_{k} = f_{k}\text{ for all }k \geqslant n.\]
\end{definition}

Remark that if $\mu_{p}$ is an AF-measure on the space of infinite
paths of the Bratteli diagram $B=(V,E)$, the tail equivalence
$\mathcal{R}_{\Omega}$ on $\Omega$ is a countable measured
equivalence relation.

For $n\geqslant1$, let $k_{n}=1+5^{n}$ and $\varphi_{n}(\cdot)$ be
the state on the the $k_{n}\times k_{n}$ matrices,
$M_{k_{n}}(\mathbb{C})$, given by
$\varphi_{n}(\cdot)=\text{tr}(h_{n}\cdot)$, where $h_{n}$ is the
diagonal matrix
\[\text{diag}\left(\frac{1}{2},\frac{1}{2\cdot 5^{n}},\frac{1}{2\cdot
5^{n}},\cdots,\frac{1}{2\cdot 5^{n}}\right).\] We consider the
Araki-Woods factor of type $III_{0}$,
$M=\otimes(M_{k_{n}}(\mathbb{C}),\varphi_{n})$. For each $n\geqslant
1$, let $u_{n}=\text{diag}(1,-1,-1, . . . ,-1)$ and let $\alpha=
\otimes\text{Ad }u_{n}$ be the resulting involutive automorphism of
$M$. We denote by $M^{\alpha}$ the fixed point algebra, which is a
subfactor of index $2$ of $M$. In \cite{GH}, it is showed that
$M^{\alpha}$ is not an ITPFI factor.

Note that the fixed point algebra $M^{\alpha}$ is an AFD factor and
$M$ can be written explicitly as the weak closure of an increasing
union of finite dimensional von Neumann algebras $M_{n}$,
$n\geqslant 1$. Let $B = (V,E)$ be the Bratteli diagram
corresponding to $\cup_{n\geqslant 1}M_{n}$. We describe this
Bratteli diagram using a slightly different notation than in
\cite{GH}, which is more convenient in this paper.

We introduce the following notation. For $n\geqslant 1$, let
$X_{n}=\{0,1,\ldots, 5^{n}\}$, $Y_{n}=\{0,1\}$, and $\pi_{n}:
X_{n}\rightarrow Y_{n}$, given by
\[\pi_{n}(0)=0, \ \ \pi_{n}(i)=1\text{ if }i\neq 0.\]
The diagram has a single vertex $v_{0,0}$ at level $0$ and, for
$n\geqslant 1$, there are two vertices at level $n$ indexed with
$v_{n,0}$ (left vertex at level $n$) and $v_{n,1}$ (right vertex at
level $n$). Let $V_{n} = \{v_{n,0}, v_{n,1}\}$, for $n\geqslant 1$.
There is an edge, indexed with $ e_{0,0,1}^{0}$, from $v_{0,0}$ to
the vertex $v_{1,0}$, and $5$ edges from $v_{0,0}$ to $v_{1,1}$,
indexed with $e_{0,1,1}^{1},\ldots e_{0,1,1}^{5}$. For $n\geqslant
2$, there exist an edge indexed with $e_{0,0,n}^{0}$, from
$v_{n-1,0}$ to $v_{n,0}$, an edge indexed with $e_{1,1,n}^{0}$, from
$v_{n-1,1}$ to $v_{n,1}$, $5^{n}$ edges indexed with $e_{0,1,n}^{1},
\ldots e_{0,1,n}^{5^{n}}$, from $v_{n-1,0}$ to $v_{n,1}$, and
$5^{n}$ edges index with $e_{1,0,n}^{1}, \ldots e_{1,0,n}^{5^{n}}$,
from $v_{n-1,1}$ to $v_{n,0}$. Let $E_{n}$ be the set of edges going
the from level $n-1$ to the level $n$. With the above notation,
$E_{1}=\{e_{0,0,n}^{0}, e_{0,1,n}^{1},\ldots e_{0,1,n}^{5}\}$ and,
for $n\geqslant 2$, $E_{n} =\{e_{0,0,n}^{0}, e_{0,1,n}^{1}, \ldots
e_{0,1,n}^{5^{n}},e_{1,1,n}^{0}, e_{1,0,n}^{1}, \ldots
e_{1,0,n}^{5^{n}} \}$. $E_{1}$ is the set of those
$e^{x_{1}}_{0,v_{1},1}$ for which $x_{1}\in\{1,2,\ldots 5\}$ if
$v_{1}=1$ and $x_{1}=0$ if $v_{1}=0$. For $n\geqslant 2$, we can
write $E_{n}$ as the set of all $e^{x_{n}}_{v_{n-1},v_{n},n}$, where
$v_{n},v_{n-1}\in\{0,1\}$ and $x_{n}\in X_{n}$ are given as follows:
$x_{n}\in \{1,2,\ldots, 5^{n}\}$ if $v_{n}\neq v_{n-1}$, and
$x_{n}=0$ if $v_{n}=v_{n-1}$. Equivalently, $E_{1}$ is the set of
all $e^{x_{1}}_{0,v_{1},1}$, where $v_{1}\in\{0,1\}$,
$x_{1}\in\{1,2,\ldots 5\}$ are such that $\pi_{1}(x_{1})=0\text{
(mod 2)}$ and for $n\geqslant 2$, $E_{n}$ is the set of all
$e^{x_{n}}_{v_{n-1},v_{n},n}$, where $v_{n},v_{n-1}\in\{0,1\}$,
$x_{n}\in X_{n}$ are such that $v_{n}-v_{n-1}=\pi_{n}(x_{n})\text{
(mod 2)}$. The paths space of the diagram can be written as the
subset of all elements in $\prod_{n\geqslant 1} E_{n}$ of the form
$(e_{0,v_{1},1}^{x_{1}}, e_{v_{1},v_{2},2}^{x_{2}}, \ldots,
e_{v_{n-1},v_{n},n}^{x_{n}},\ldots)$, where $\pi_{1}(x_{1})=v_{1}$
and $\pi_{n}(x_{n})=v_{n}-v_{n-1}\text {(mod 2)}$ for $n\geqslant
2$. Hence, if $\Omega$ denotes the paths space of the Bratteli
diagram, then
\begin{align*}
\Omega=&\{(e_{0,v_{1},1}^{x_{1}}, e_{v_{1},v_{2},2}^{x_{2}}, \ldots,
e_{v_{n-1},v_{n},n}^{x_{n}},\ldots),\text{ where }x_{n}\in
X_{n}\text{ and }\\
& v_{n}=\pi_{1}(x_{1})+\pi_{2}(x_{2})+\cdots+\pi_{n}(x_{n})\text{
(mod 2)}\text{ for }n\geqslant1\}.
\end{align*}
Let $\mathcal{R}_{\Omega}$ be the tail equivalence relation on
$\Omega$. Two paths in $\Omega$, $(e_{0,v_{1},1}^{x_{1}},
e_{v_{1},v_{2},2}^{x_{2}},\ldots$

\noindent $\ldots, e_{v_{n-1},v_{n},n}^{x_{n}},\ldots)$ and
$(e_{0,u_{1},1}^{y_{1}}, e_{u_{1},u_{2},2}^{y_{2}}, \ldots,
e_{u_{n-1},u_{n},n}^{y_{n}},\ldots)$ are tail equivalent if and only
if
\begin{align*}
\text{for some }n\geqslant 1, \ x_{i}=y_{i}\text{ if }i>n\text{ and\
} \sum_{i=1}^{n}\pi_{i}(x_{i})-\pi_{i}(y_{i})=0\text{ (mod 2)}.
\end{align*}
On E, we define the system $p$ of transition probabilities given by
\[p_{1} : E_{1}\rightarrow [0, 1] \ \  p_{1}(e_{0,0,1}^{0}) =\frac{1}{2},\ \  p_{1}(e_{0,1,1}^{i}) = \frac{1}{2\cdot 5}, \ 1\leqslant i\leqslant 5\] and,
for $n\geqslant 2$,
\[p_{n}: E_{n}\rightarrow [0, 1]\ \  p_{n}(e_{0,0,n}^{0})= p_{n}(e_{1,1,n}^{0}) =\frac{1}{2},\ \
p_{n}(e_{0,1,n}^{i})=p_{n}(e_{1,0,n}^{i})=\frac{1}{2\cdot 5^{n}},\
i\neq 0.\] Let $\mu_{p}$ be the AF-measure on $\Omega_{0}$ defined
by
\[\mu_{p}(f)=p_{1}(e_{0,v_{1},1}^{x_{1}})\cdot p_{2}(e_{v_{1},v_{2},2}^{x_{1}})\cdots p_{n}(e_{v_{n-1}v_{n},n}^{x_{n}}),\]
where $f$ is the cylinder set
$(e_{0,v_{1},1}^{x_{1}},e_{v_{1},v_{2},2}^{x_{1}},\ldots,e_{v_{n-1},v_{n},n}^{x_{n}})$
of length $n$ ($v_{i}\in\{0,1\}$, and $x_{i}\in\{0,1,\ldots,5^{i}\}$
such that $v_{n}-v_{n-1}=\pi_{n}(x_{n})\text{ (mod 2)}$ for
$n\geqslant 1$ and $\pi_{1}(x_{1})=v_{1}$). According to \cite{GH},
we have that $M^{\alpha}\simeq M(\Omega,
\mathcal{R}_{\Omega},\mu_{p})$.

For $n\geqslant 1$, we consider $\mu_{n}$ the measure on $X_{n}$
defined as
\begin{align*}
\mu_{n}(0)=\frac{1}{2}, \ \ \ \ \ \ \ \mu_{n}(i)=\frac{1}{2\cdot
5^{n}}\ \ \ 1\leqslant i \leqslant 5^{n}.
\end{align*}
Let $X=\prod_{n\geqslant 1} X_{n}$ endowed with the product measure
$\mu=\otimes\mu_{n}$ and let $\mathcal{R}$ be the equivalence
relation on $X$, given by:
\begin{align*}
x\mathcal{R} y \text{ if for some  }n\geqslant 1, \
x_{i}=y_{i}\text{ for all }i>n\text{ and
 }\sum\limits_{i=1}^{n}\pi_{i}(x_{i})-\pi_{i}(y_{i})=0\text{ (mod 2)}.
\end{align*}
Further, we define $S: \Omega\rightarrow X$. For
$x=(e_{0,v_{1},1}^{x_{1}}, e_{v_{1},v_{2},2}^{x_{2}}, \ldots,
e_{v_{n-1},v_{n},n}^{x_{n}},\ldots)\in \Omega$, we set
$Sx=(x_{1},x_{2},\ldots,x_{n},\ldots)\in X$. We immediately observe
that $S$ is an orbit equivalence between
$(\Omega,\mu_{p},\mathcal{R}_{\Omega})$ and $(X,\mu,\mathcal{R})$.
Moreover, $\mu\circ S=\mu_{p}$.

Since $\mathcal{R}$ is hyperfinite, a striking result of Connes,
Feldman and Weiss \cite{CFW} says that $\mathcal{R}$ is amenable and
there exists a non-singular transformation $T$ of $X$ such that, up
to a null set, $\mathcal{R}=\mathcal{R}_{T}$. Even though this will
be enough to prove the main result of this paper, we can define
explicitly such a non-singular transformation as follows.

For $\mu$-a.e. $x\in X$, let $N(x)=\min\{i\geqslant 2:
x_{i}<5^{i}\}<\infty$ and $a_{x}\in\{0,1\}$ such that
$a_{x}=\pi_{1}(x_{1})+\pi_{2}(x_{2})+\cdots+\pi_{N(x)}(x_{N(x)})-1\text{
(mod 2)}$. We can define a non-singular transformation $T$ on
$(X,\mu)$ by setting, for almost all $x\in X$,
$$T(x_{1},x_{1},\ldots, x_{n},\ldots)=\left\{
\begin{array}{cl}
(x_{1}+1,x_{2},\ldots,x_{n}, x_{n+1},\ldots)
\text{ if }x_{1}\in\{1,2,3,4\},\\[0.3cm]
(a,0,\ldots,x_{N(x)}+1,x_{N(x)+1},\ldots) \text{ if
}x_{1}\in\{0,5\}.
\end{array}
\right.
$$
Clearly, $\mathcal{R}(x)=\mathcal{R}_{T}(x)$ for almost all $x\in
X$. Hence $\mathcal{R}$ is orbit equivalent to $\mathcal{R}_{T}$
and, consequently, $\mathcal{R}_{\Omega}$ and $\mathcal{R}_{T}$ are
orbit equivalent. $T$ is an ergodic non-singular transformation,
since $M\simeq M(X,\mu,\mathcal{R}_{T})$ is a factor of type
$III_{0}$. From \cite{GH}, we have that the associated flow of
$\mathcal{R}_{T}$ is not AT and therefore, by \cite{CW}, it follows
that $T$ is not of product type.

\section{T has Krieger's Propperty A}
\bigskip

In this section, we show that the previously defined $T$ has
Krieger's property A. The cylinders sets in $X$ are denoted as
follows:
\[Z_{A}=\{x\in X; (x_{i})_{i\in \Gamma}\in A\}, \ \ \Gamma\subseteq \mathbb{N}^{*}, \ A\subseteq\prod_{i\in \Gamma}X_{i}.\]
We define random variable $a_{\mu,j}(x)$ on $(X,\mu)$ by setting,
\[a_{\mu,j}(x)=-\log \mu_{j}(x_{j}),\  x\in X, \ j\geqslant1.\]
First, we proof  a technical result.
\begin{lemma}\label{le}
Fix $\xi>0$, $I\in \mathbb{N}^{*}$, $E\subseteq\{v\in
\prod_{i=I+1}^{N}X_{i}\}$, $u\in\prod_{i=1}^{I}X_{i}$ and
$B\subseteq Z_{u}$ such that $\mu(B)>(1-\frac{\xi}{128})\mu(Z_{u})$
and $\mu(Z_{E})>\frac{\xi}{16}$. Let $E^{0}=\{v\in E:\mu(B\cap
Z_{v})>\frac{3}{4}\mu(Z_{u})\mu(Z_{v})$. Then
 $\mu(Z_{E^{0}})>\frac{1}{2}\mu(Z_{E})$.
\end{lemma}

\begin{proof}Suppose by contradiction that
$\mu(Z_{E^{0}})\leqslant\frac{1}{2}\mu(Z_{E})$. Then
\[\mu(Z_{E-E^{0}})\geqslant\frac{1}{2}\mu(Z_{E}).\]
If $v\in E-E^{0}$ then $\mu(Z_{v}\cap B)\leqslant
\frac{3}{4}\mu(Z_{u})\mu(Z_{v})$. It follows that
\[\mu(B\cap Z_{E-E^{0}})\leqslant\frac{3}{4}\mu(Z_{u})\mu(Z_{E-E^{0}}).\]
Therefore \[\mu(Z_{E-E^{0}}\cap
(Z_{u}-B))\geqslant\frac{1}{4}\mu(Z_{u})\mu(Z_{E-E^{0}}),\]which
implies
\[\mu(Z_{u}-B)\geqslant\mu(Z_{E-E^{0}}\cap
(Z_{u}-B))\geqslant\frac{1}{4}\mu(Z_{u})\mu(Z_{E-E^{0}})\geqslant\frac{\xi}{128}\mu(Z_{u}).\]
Hence
\[\mu(Z_{u})(1-\frac{\xi}{128})\geqslant \mu(B),\]
which is a contradiction.
\end{proof}

\noindent We can give now the proof of the main result.
\begin{theorem}\label{giord}
If $T$ is the non-singular transformation defined in the previous
section, then $T$ has Krieger's property A.
\end{theorem}

\begin{proof} The first part of the proof is similar to Krieger's proof,
\cite{K3}, for a product odometer. From \cite{K3}, there exist
sequences $b(i)>0$, $c(i)\in\mathbb{R}$, and $0<\beta < \frac{1}{2}$
with the following properties:
\begin{align}
&\ \ \ \ \ \ \ \ \ \ \ \ \ \ \ \ \ \ \sup\limits_{i\geqslant 1}b_{i}=\infty, \label{1}\\
&\mu\{x\in X :|c(i)+\sum\limits_{j=1}^{i}a_{\mu,j}(x)|\leqslant
b(i)\}\geqslant 1-2\beta,
\label{2}\\
&\mu\{x\in X : \ c(i)+\sum\limits_{j=1}^{i}a_{\mu,j}(x)\geqslant
b(i)\}\geqslant\beta, \label{3}\\
&\mu\{x\in X : \ c(i)+\sum\limits_{j=1}^{i}a_{\mu,j}(x)\leqslant
-b(i)\}\geqslant\beta. \label{4}
\end{align}
Let
\[\xi=\min\{\beta,1-2\beta\}.\]
Using (\ref{1}), we can choose a sequence $i(k)_{k\in \mathbb{N}}$
such that
\begin{equation}\label{5}
\lim\limits_{k\rightarrow \infty}b(i(k))=\infty
\end{equation}
with the property that, as $k\rightarrow\infty$, the random
variables
\[b(i(k))^{-1}(c(i(k))+\sum\limits_{j=1}^{i(k)}a_{\mu,j})\]
converge in distribution with a limit measure denoted by $\lambda$
(by Helly's Selection Theorem, [B], Theorem 29.3). From (\ref{2}),
we obtain
\begin{equation}\label{6}
\lambda([-1,1])\geqslant\xi.
\end{equation}
We choose $\rho$, $|\rho|\leqslant 1$, such that
\begin{equation}\label{7}
\lambda\left(\rho-\frac{1}{2},\rho+\frac{1}{2}\right)\geqslant
\frac{\xi}{3}.
\end{equation}
Let $A\subseteq X$. There exists $I\in\mathbb{N}$ and
$u\in\prod_{i=1}^{I}X_{i}$ such that
\[\mu(B)>(1-\frac{\xi}{128})\mu(Z_{u}),\]
where $B=A\cap Z_{u}$. The random variables
\begin{equation}\label{7bis}
b(i(k))^{-1}(c(i(k))+\sum\limits_{j=I+1}^{i(k)}a_{p,j})
\end{equation}
also converge in distribution and the limit measure is also
$\lambda$.

\noindent Let $M>1$ such that
\begin{equation}\label{7trei}
e^{M-2}>(I+1)\log 5.
\end{equation}
From (\ref{5}) and (\ref{7}), we can choose $k\in \mathbb{N}$ such
that
\begin{align}
b(i(k))>e^{M+1}, \label{8}\\
b(i(k))>-4\sum\limits_{j=1}^{I}\min_{x\in X_{j}}
\mu_{j}(x)\label{9},
\end{align}
and setting
\[\Gamma=\{v\in\prod_{i=I+1}^{i(k)}X_{i} :(\rho-\frac{1}{2})b(i(k))\leqslant c(i(k))-\sum\limits_{j=I+1}^{i(k)}\log \mu_{j}(v_{j})\leqslant
(\rho+\frac{1}{2})b(i(k))\},\] we have
\begin{equation}\label{10}
\mu(Z_{\Gamma})>\frac{\xi}{4}.
\end{equation}
Let
\begin{align*}
\Gamma_{-}&=\{v\in\prod_{i=I+1}^{i(k)}X_{i}
:c(i(k))-\sum\limits_{j=I+1}^{i(k)}\log \mu_{j}(v_{j})\leqslant
-\frac{3}{4}b(i(k))\},\\
\Gamma_{+}&=\{v\in\prod_{i=I+1}^{i(k)}X_{i}
:c(i(k))-\sum\limits_{j=I+1}^{i(k)}\log \mu_{j}(v_{j})\geqslant
\frac{3}{4}b(i(k))\}.
\end{align*}
 From (\ref{3}), (\ref{4}) and (\ref{9}) we
have
\begin{equation}\label{11}
\mu(Z_{\Gamma_{-}})\geqslant\xi \ \ \ \ \ \ \
\mu(Z_{\Gamma_{+}})\geqslant\xi.
\end{equation}
Using Lemma \ref{le}, we find
\begin{equation}\label{12}
v^{-}\in\Gamma_{-} \text{ and } v^{+}\in\Gamma_{+},
\end{equation}
such that
\begin{equation}\label{13}
\mu(B\cap Z_{v^{-}})>\frac{3}{4}\mu(Z_{u})\mu(Z_{v^{-}}),
\end{equation}
\begin{equation*}
\mu(B\cap Z_{v^{+}})>\frac{3}{4}\mu(Z_{u})\mu(Z_{v^{+}}).
\end{equation*}For the
remaining part of the proof we argue in a different manner than
Krieger did in the case of a product odometer.

Assume that $\rho\geqslant0$. We define:
\[V^{0}=\{v\in\prod_{j=I+1}^{i(k)}X_{i}:\sum\limits_{j=I+1}^{i(k)}
\pi_{j}(v_{j})-\pi_{j}(v^{-}_{j})=0\text{ (mod 2)}\},\]
\[V^{1}=\{v\in\prod_{j=I+1}^{i(k)}X_{i}:\sum\limits_{j=I+1}^{i(k)}
\pi_{j}(v_{j})-\pi_{j}(v^{-}_{j})=1\text{ (mod 2)}\},\] and
\begin{multline}\label{14}
\Psi=\left\{ v\in\prod_{i=I+1}^{i(k)}X_{i}
:(\rho-\frac{1}{2})b(i(k))-(I+1)\log 5\leqslant\right.\\
\left.\leqslant c(i(k))-\sum\limits_{j=I+1}^{i(k)}\log
\mu_{j}(v_{j})\leqslant (\rho+\frac{1}{2})b(i(k))+(I+1)\log 5\
\right\}.
\end{multline}
Let
\begin{equation}\label{15}
s=\log|c(i(k))-\sum\limits_{j=I+1}^{i(k)}\log \mu_{j}(v_{j}^{-})|.
\end{equation} Since $v^{-}\in\Gamma_{-}$, we have:
\begin{equation}\label{16}
e^{s}\geqslant\frac{3}{4}b(i(k))
\end{equation}
and from (\ref{8})
\begin{equation}
s>M\label{17}.
\end{equation}

\noindent We will show that there exists a set $E\subseteq\Psi\cap
V^{0}$ such that $\mu(Z_{E})>\frac{\xi}{16}$.

\noindent From (\ref{10}), we have that
$\mu(Z_{\Gamma})>\frac{\xi}{4}$. Note that there are two
possibilities:

\noindent\textbf{Case  I: }  If $\mu(Z_{\Gamma\cap
V^{0}})>\frac{\xi}{16}$, we define $E=\Gamma\cap V^{0}\subseteq
\psi\cap V^{0}$.

\noindent\textbf{Case II.} If $\mu(Z_{\Gamma\cap
V^{0}})\geqslant\frac{\xi}{16}$ we have by (\ref{10}) that
$\mu(Z_{\Gamma\cap V^{1}})>\frac{3\xi}{16}$.

\noindent Without loss of generality, we assume that $v^{-}$
satisfies $\sum\limits_{i=I+1}^{i(k)}\pi_{i}(v^{-}_{i})=0$ (mod 2)
(if $\sum\limits_{i=I+1}^{i(k)}\pi_{i}(v^{-}_{i})=1$ (mod 2) we
proceed in a similar way).

We consider the following sets:
\begin{align*}
&A_{0}^{0}=\{v\in \Gamma; v_{I+1}=0,
\sum_{i=I+2}^{i(k)}\pi_{i}(v_{i})=0\text{ (mod 2)}\},\\
&A_{1}^{0}=\{v\in \Gamma; \pi_{I+1}(v_{I+1})=1,
\sum_{i=I+2}^{i(k)}\pi_{i}(v_{i})=0\text{ (mod 2)}\},\\
&A_{0}^{1}=\{v\in \Gamma; v_{I+1}=0,
\sum_{i=I+2}^{i(k)}\pi_{i}(v_{i})=1\text{ (mod 2)}\},\\
&A_{1}^{1}=\{v\in \Gamma; \pi_{I+1}(v_{I+1})=1,
\sum_{i=I+2}^{i(k)}\pi_{i}(v_{i})=1\text{ (mod 2)}\}.
\end{align*}and
\begin{align*}
&B_{1}^{1}=\{x\in \prod_{i=I+1}^{i(k)}X_{i};\exists v\in
A_{0}^{1}\text{ with } v_{i}=x_{i}, i\geqslant I+2,
\text{ and }x_{I+1}\in\{1,\ldots 5^{I}\}\}\\
&B_{0}^{0}=\{x\in \prod_{i=I+1}^{i(k)}X_{i};\exists v\in
A_{1}^{0}\text{ with } v_{i}=x_{i}, \text{ for }i\geqslant I+2,
\text{ and }x_{I+1}=0\}
\end{align*}
We immediately observe that
\[\Gamma\cap V^{0}=A_{0}^{0}\cup
A_{1}^{1},\ \ \Gamma\cap V^{1}=A_{0}^{1}\cup A_{1}^{0},\]
\[\mu(Z_{B^{0}_{0}})\geqslant\mu(Z_{A^{0}_{1}})\text{ and
}\mu(Z_{B^{1}_{1}})\geqslant\mu(Z_{A^{1}_{0}}).\] As
$\mu(Z_{V^{1}\cap\Gamma})>\frac{3\xi}{16}$ and $\mu(B_{0}^{0}\cup
B_{1}^{1})\geqslant \mu(A_{0}^{1}\cup A_{1}^{0})$, we obtain
\begin{equation}\label{18bis}
\mu(Z_{B_{0}^{0}\cup B_{1}^{1}})>\frac{3\xi}{16}>\frac{\xi}{16}.
\end{equation}
\noindent We claim that
\begin{equation}\label{18trei}
B_{0}^{0}\cup B_{1}^{1}\subseteq \Psi.
\end{equation}
First, we show that $B_{0}^{0}\subseteq\Psi$. Let $x\in B_{1}^{1}$.
From the way $B_{1}^{1}$ has been defined, we have that
$x_{I+1}\in\{1,2,\ldots 5^{I}\}$ and there exists $v\in A_{0}^{1}$
with $v_{i}=x_{i}$, for $i\geqslant I+2$. As $v\in
A^{1}_{0}\subseteq \Gamma$ and
\begin{align*}
\sum_{i=I+1}^{i(k)}(\log \mu_{i}(x_{i})-\log \mu_{i}(v_{i}))&=\log
\mu_{I+1}(x_{I+1})-\log \mu_{I+1}(v_{I+1})\\&=-\log
\frac{1}{2}+\log\frac{1}{2\cdot 5^{I+1}}\\
&=-(I+1)\log 5,
\end{align*}
it follows that $x\in \Psi$. Since
$\sum\limits_{j=I+1}^{i(k)}\pi_{j}(x_{j})-\pi_{j}(v_{j}^{-})=0$ (mod
2), it follows that $x\in V^{0}$. Similarly, if $x\in B_{0}^{0}$ we
have $x\in\Psi\cap V^{0}$.

We define $E=B_{0}^{0}\cup B_{0}^{0}$. From (\ref{18bis}) and
(\ref{18trei}), it results that $E\subseteq \Psi\cap V^{0}$ and
$\mu(Z_{E})\geqslant\frac{\xi}{16}$.

Hence, in both cases, we showed that we can find $E\subseteq
\Psi\cap V^{0}$ such that $\mu(Z_{E})>\frac{\xi}{16}$.

From Lemma \ref{le}, we obtain that there exists $E^{0}\subseteq E$
such that
\begin{equation*}
\mu(Z_{E^{0}})>\frac{\xi}{32},
\end{equation*}
\begin{equation}\label{18} \mu(B\cap
Z_{v})>\frac{3}{4}\mu(Z_{u})\mu(Z_{v}) \text{ for all }v\in E^{0}.
\end{equation}

Let $v\in E^{0}$. Since
$\sum\limits_{i=I+1}^{i(k)}\pi_{i}(v_{i})-\pi_{i}(v^{-}_{i})=0$ (mod
2), there exists $S_{v}\in [T]$ such that
\begin{align*}
S_{v}:Z_{u}\cap Z_{v^{-}}\rightarrow Z_{u}\cap Z_{v}\\
(S_{v}x)_{j}=x_{j},\ \ \ \ j>i(k).
\end{align*}
From (\ref{13}) and (\ref{18}) we have
\begin{equation}\label{19}
\mu(B\cap Z_{v}\cap S_{v}(B\cap
Z_{v^{-}}))>\frac{1}{2}\mu(Z_{u})\mu(Z_{v}),\text{ for all }v\in
E^{0}
\end{equation}
\noindent We want to prove that
\begin{equation}\label{20}
B\cap Z_{v}\cap S_{v}(B\cap Z_{v^{-}})\subseteq
K_{\mu,T}(B,s,3)\text{ for all }v\in E^{0}
\end{equation}
\noindent First we show that for $v\in E^{0}$,
\begin{equation}\label{21}
|\log(\sum_{I<j\leqslant i(k)}(\log \mu_{j}(v_{j})-\log
\mu_{j}(v_{j}^{-})))-s|<3.
\end{equation}
In order to prove (\ref{21}), from (\ref{7trei}), (\ref{14}),
(\ref{15}), (\ref{16}) and (\ref{17}), we have
\begin{align*}
\sum\limits_{j=I+1}^{i(k)}\log\mu_{j}(v_{j}^{-})-\log \mu_{j}(v_{j})
=c(i(k))-\sum\limits_{j=I+1}^{i(k)}\log
\mu_{j}(v_{j})+e^{s}\\
\leqslant (\rho+\frac{1}{2})b(i(k))+e^{s}+ (I+1)\log 5\leqslant
e^{s+2}+(I+1)\log 5 <e^{s+3}.
\end{align*}
From (\ref{7trei}), (\ref{14}), (\ref{15}), (\ref{16}), (\ref{17}),
and the fact that $\rho\geqslant 0$, we have
\begin{align*}
\sum\limits_{j=I+1}^{i(k)}(\log \mu_{j}(v_{j}^{-})-\log
\mu_{j}(v_{j})) =c(i(k))-\sum\limits_{j=I+1}^{i(k)}\log
\mu_{j}(v_{j})+e^{s}\ \ \ \ \ \ \ \ \ \ \ \\
\geqslant (\rho-\frac{1}{2})b(i(k))+e^{s}-(I+1)\log 5\geqslant
e^{s}-\frac{2}{3}e^{s}- (I+1)\log 5\\ >\frac{1}{3}e^{s}-(I+1)\log
5>e^{s-3}
\end{align*}
So, from (\ref{19}) and (\ref{20}) we conclude that:
\[\mu(K_{\mu,T}(B,s,3))>\frac{1}{2}\mu(Z_{u})\frac{\xi}{32}\geqslant\frac{\xi}{64}\mu(B)\]
If $\rho\leqslant 0$ we proceed in a similar way with $v^{+}$
instead of $v^{-}$. Hence $T$ has Krieger's property A.
\end{proof}
Therefore, we have
\begin{corollary}
There exist ergodic non-singular transformations of non-product type
which satisfy Krieger's property A.
\end{corollary}
\proof[Acknowledgements] This paper is part of the author's Ph.D.
thesis at the University of Ottawa. The author is grateful to his
Ph.D. supervisor, Professor Thierry Giordano for his guidance,
encouragements and support. The author is also grateful to dr.
Dumitru Trucu from University of Dundee, for useful suggestion in
writing this paper.

\end{document}